\documentclass{amsart}

\usepackage{psfrag}
\usepackage{color}
\usepackage{tikz}
\usetikzlibrary{matrix,arrows}
\usepackage{graphicx,graphics}
\usepackage{amssymb,amsfonts,amsmath,amstext,amsthm,amscd,verbatim,enumerate,soul}
\usepackage[T1]{fontenc}

\begin{document}

\newtheorem{theorem}{Theorem}[section]
\newtheorem{result}[theorem]{Result}
\newtheorem{fact}[theorem]{Fact}
\newtheorem{conjecture}[theorem]{Conjecture}
\newtheorem{lemma}[theorem]{Lemma}
\newtheorem{proposition}[theorem]{Proposition}
\newtheorem{corollary}[theorem]{Corollary}
\newtheorem{facts}[theorem]{Facts}
\newtheorem{props}[theorem]{Properties}
\newtheorem*{thmA}{Theorem A}
\newtheorem{ex}[theorem]{Example}
\theoremstyle{definition}
\newtheorem{definition}[theorem]{Definition}
\newtheorem*{remark}{Remark}
\newtheorem{example}[theorem]{Example}
\newtheorem*{defna}{Definition}

\newcommand{\notes} {\noindent \textbf{Notes.  }}
\newcommand{\defn} {\noindent \textbf{Definition.  }}
\newcommand{\defns} {\noindent \textbf{Definitions.  }}
\newcommand{\x}{{\bf x}}
\newcommand{\e}{\epsilon}
\renewcommand{\d}{\delta}
\newcommand{\z}{{\bf z}}
\newcommand{\B}{{\bf b}}
\newcommand{\V}{{\bf v}}
\newcommand{\T}{\mathbb{T}}
\newcommand{\Z}{\mathbb{Z}}
\newcommand{\Hp}{\mathbb{H}}
\newcommand{\D}{\Delta}
\newcommand{\R}{\mathbb{R}}
\newcommand{\N}{\mathbb{N}}
\renewcommand{\B}{\mathbb{B}}
\renewcommand{\S}{\mathbb{S}}
\newcommand{\C}{\mathbb{C}}
\newcommand{\ft}{\widetilde{f}}
\newcommand{\dt}{{\mathrm{det }\;}}
 \newcommand{\adj}{{\mathrm{adj}\;}}
 \newcommand{\0}{{\bf O}}
 \newcommand{\av}{\arrowvert}
 \newcommand{\zbar}{\overline{z}}
 \newcommand{\xbar}{\overline{X}}
 \newcommand{\htt}{\widetilde{h}}
\newcommand{\ty}{\mathcal{T}}
\newcommand\diam{\operatorname{diam}}
\renewcommand\Re{\operatorname{Re}}
\renewcommand\Im{\operatorname{Im}}
\newcommand{\tr}{\operatorname{Tr}}
\renewcommand{\skew}{\operatorname{skew}}

\newcommand{\ds}{\displaystyle}
\numberwithin{equation}{section}
%
%
%
%
\newcommand{\M}{{\mathcal{M}}}
\newcommand{\tef}{transcendental entire function}
\newcommand{\qfor}{\quad\text{for }}
\newcommand*{\defeq}{\mathrel{\vcenter{\baselineskip0.5ex \lineskiplimit0pt
 \hbox{\scriptsize.}\hbox{\scriptsize.}}}
 =}
%
%

\renewcommand{\theenumi}{(\roman{enumi})}
\renewcommand{\labelenumi}{\theenumi}

\newcommand{\alastair}[1]{{\scriptsize \color{red}\textbf{Alastair's note:} #1 \color{black}\normalsize}}

\title[The maximum modulus set]{The maximum modulus set of a quasiregular map}

\author{Alastair N. Fletcher}
\address{Department of Mathematical Sciences, Northern Illinois University, DeKalb, IL 60115-2888. USA}
\email{fletcher@math.niu.edu}

\author{David J. Sixsmith}
\address{School of Mathematics and Statistics\\ The Open University\\
Milton Keynes MK7 6AA\\ UK\textsc{\newline \indent {\includegraphics[width=1em,height=1em]{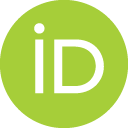} {\normalfont https://orcid.org/0000-0002-3543-6969}}}}
\email{david.sixsmith@open.ac.uk}

\thanks{The first author was supported by a grant from the Simons Foundation, \#352034}

\date{\today}

\begin{abstract}
We study, for the first time, the maximum modulus set of a quasiregular map. It is easy to see that these sets are necessarily closed, and contain at least one point of each modulus. Blumenthal showed that for entire maps these sets are either the whole plane, or a countable union of analytic curves. We show that in the quasiregular case, by way of contrast, any closed set containing at least one point of each modulus can be attained as the maximum modulus set of a quasiregular map. These examples are all of polynomial type. We also show that, subject to an additional constraint, such sets can even be attained by quasiregular maps of transcendental type.
\end{abstract}

\maketitle

\section{Introduction}
Suppose first that $f$ is an entire function, and define the \emph{maximum modulus function} by
\begin{equation}
\label{Mfdef}
M(r, f) \defeq \max_{|z| = r} |f(z)|, \qfor r \geq 0.
\end{equation}
Following \cite{Sixsmithmax}, we denote the set of points where $f$ achieves its maximum modulus, known as the \emph{maximum modulus set}, by $\M(f)$. In other words, we have
\begin{equation}
\label{Mdef}
\M(f) \defeq \{ z \in \C : |f(z)| = M(|z|,f) \}.
\end{equation} 

The set $\M(f)$ was first studied by Blumenthal \cite{Blum}. He showed, see \cite[Theorem 10]{Val}, that if $f$ is a monomial, then $\M(f) = \C$, and otherwise $\M(f)$ consists of an, at most countable, union of closed curves, each of which is analytic except at  its endpoints. A number of authors have since studied entire functions for which $\M(f)$ has unusual and pathological properties; see, for example, \cite{Hardy, PSS, PSS2, Tyler}.

The natural generalisation of analytic maps to higher dimensional spaces is the widely studied class of \emph{quasiregular} maps. Roughly speaking the difference between these classes is as follows. Whereas an analytic function maps infinitesimal circles to infinitesimal circles, a quasiregular function maps infinitesimal spheres to infinitesimal ellipsoids, with a uniform upper bound on the distortion of the ellipsoids. This weaker condition means that the class of quasiregular maps is much less rigid than that of analytic maps; for example, in (real) dimension greater than two there are many families of quasiregular maps, but the M\"obius maps are the only conformal maps. It is very natural, then, to ask how Blumenthal's result extends to quasiregular maps, and this is our goal here. 

In the interests of brevity, we refer to \cite{Rickman, Vuorinen} for the precise definition and basic properties of quasiregular maps. 
We say that a quasiregular map is \emph{$K$-quasiregular}, if its maximal dilatation is bounded above by some $K\geq 1$. A quasiregular map is of \emph{polynomial type} if it has finite degree, and otherwise it is of \emph{transcendental type}. A \emph{quasiconformal map} is an injective quasiregular map.

For a quasiregular map $f : \R^n \to \R^n$ the definition \eqref{Mfdef} carries across without modification, and the definition \eqref{Mdef} becomes simply
\begin{equation}
\label{Mdefqr}
\M(f) \defeq \{ x \in \R^n : |f(x)| = M(|x|,f) \}.
\end{equation}
Since the class of quasiregular maps is somewhat less rigid than the class of entire maps, one might expect that a result similar to Blumenthal's, but weaker, applies in this setting. It turns out that, in fact, the sharpest possible result is achievable.
\begin{theorem}
\label{thm:1}
Let $n \geq 2$, and let $T \subset \R^n$ be a closed set which meets every sphere centred at the origin of radius $r \geq 0$. Then, for each $d \in \N$, there exists a quasiregular map $h:\R^n \to \R^n$, of polynomial-type and of degree $d^{n-1}$, for which $\mathcal{M}(h) = T$.
\end{theorem}

For transcendental-type quasiregular maps, Theorem \ref{thm:1} cannot hold as stated, for the following reasons. Suppose that $f : \R^n \to \R^n$ is a quasiregular map of transcendental type. It is well-known that the image under $f$ of every neighbourhood of infinity can only omit a finite set. It follows that $\M(f)$ fails to contain a neighbourhood of infinity. 

However, we show the following. Roughly speaking, this result shows that any closed set that contains at least one point of each modulus can be arbitrarily well approximated by the maximum modulus set of a quasiregular map of transcendental type. Here, for $r \geq 0$, we denote by $S(r)$ the sphere of radius $r$; in other words, $S(r) \defeq \{x \in \R^n : |x| = r \}$.

\begin{theorem}
\label{thm:2}
Let $n \geq 2$, and let $T \subset \R^n$ be a closed set which meets every sphere centred at the origin of radius $r \geq 0$. Then the following holds. Suppose that $(r_n)_{n\in \N}$ is a sequence of positive real numbers tending to infinity,
and that $\epsilon \in (0,1)$. Define the exceptional sets
\begin{equation}
\label{Eeq}
E_{\epsilon} \defeq \bigcup_{n \in \N} (\epsilon r_n,r_n), \quad\text{and}\quad S_{\epsilon} \defeq \bigcup_{r\in E_{\epsilon}} S(r).
\end{equation}
Then there exists a transcendental-type quasiregular map $h:\R^n \to \R^n$ with the property that $T \setminus S_{\epsilon} = \mathcal{M}(h) \setminus S_{\epsilon}$.
\end{theorem}

\begin{remark}\normalfont
By choosing the sequence $(r_n)_{n \in \N}$ to grow sufficiently quickly, for example by setting $r_n =~\exp(e^n)$, we can ensure that the exceptional set $E_{\epsilon}$ has upper density zero, or even upper logarithmic density zero.
\end{remark}

\subsection*{Notation}
Fix $n\geq 2$. For $t\in \R$ we denote by $H_t$ the hyperplane 
\[
H_t \defeq \{ (x_1,\ldots, x_{n-1},t) : x_1,\ldots, x_{n-1} \in \R \}.
\]
For $x'\in \R^{n-1}$, we denote by $\lambda_{x'}$ the line 
\[
\lambda_{x'} \defeq \{ (x',t) : t\in \R \}.
\]
For $x\in \R^n$ and $r \geq 0$, the sphere centred at $x$ of radius $r$ is denoted by $S(x,r)$, and, as noted earlier, if $x=0$, we may use $S(r)$ for brevity. (Note that it is convenient to consider a singleton as a (degenerate) sphere of radius zero.) We also will need to use the maximum norm on $\R^n$ given by
\[ ||x||_{\infty} \defeq \max_{i\in \{1,\ldots, n\}} |x_i|,\]
with corresponding sphere 
\[ Q(r) \defeq \{ x\in \R^n : ||x||_{\infty} = r \}.\]
Finally, for $x\in \R^n$, we denote by $x_n$ the $n$-th coordinate of $x$. 

\subsection*{Acknowledgments}
We would like to thank Rod Halburd for organizing the \emph{CAvid} online seminar series, out of which this note arose.

%
%
\section{Certain quasiregular mappings}
Various classes of quasiregular maps have been constructed in the literature. We briefly recall three such maps that are central to our constructions.

\subsection{Zorich mappings}
\label{sec:Z}
The class of Zorich mappings provides a well-known generalisation of the exponential function in the plane. 
These maps are strongly automorphic with respect to a discrete group $G$ of isometries. This means that if $\mathcal{Z}$ is such a map, then $\mathcal{Z}(g(x)) = \mathcal{Z}(x)$ for all $g\in G $ and all $x\in \R^n$ and, moreover, if $\mathcal{Z}(x) = \mathcal{Z}(y)$ then $y=g(x)$ for some $g\in G$.

For each $n \geq 2$, we will fix one particular Zorich map, which we denote simply by $\mathcal{Z}$, with the following properties. Here, for a quasiregular map $f$, we denote by $\mathcal{B}_f$ the \emph{branch set} of $f$; in other words the set of points at which $f$ is not locally a homeomorphism.
\begin{enumerate}[(i)]
\item For each $t \in \R$,  $\mathcal{Z}$ maps the hyperplane $H_t$ onto the sphere $S(e^t)$.\label{spheres}
\item The group $G$ is generated by a translation subgroup of rank $n-1$ and a finite group of rotations about the origin such that each $g\in G$ preserves the $n$-th coordinate.
\item A fundamental set $B$ for the action of $G$ has closure given by a beam $\Omega \times \R$, where $\Omega \subset \R^{n-1}$ is a cuboid.
\item The branch set $\mathcal{B}_{\mathcal{Z}}$ is equal to the edges of the beam $\Omega \times \R$, and so consists of $(n-2)$-dimensional hyperplanes.\label{branch}
\end{enumerate}

For an explicit formula for such a Zorich map, together with its associated group $G$, we refer to, for example, \cite[Section 3]{FP}.

\subsection{Power mappings}
\label{sec:power}
The Zorich mappings can be used to generalise the power maps. Let $A(x) \defeq dx$ for a positive integer $d \geq 2$. Then, with $\mathcal{Z},G$ as above, we have $AGA^{-1} \subset G$. It follows that the Schr\"oder equation 
\[
P\circ \mathcal{Z} = \mathcal{Z} \circ A
\]
has a unique quasiregular solution $P$ (this solution is actually uniformly quasiregular, but we do not need this fact here). This solution is a quasiregular analogue of the power map $z \mapsto z^d$ in the plane, since in this case the Schr\"oder equation is just $P(e^z) = e^{dz}$. 

These quasiregular power maps were constructed by Mayer \cite{Mayer} and have the following properties:
\begin{enumerate}[(i)]
\item $P$ is of degree $d^{n-1}$, and hence is of polynomial-type.
\item For each $r \geq 0$, $P$ maps the sphere $S(r)$ onto the sphere $S(r^d)$.
\end{enumerate}

\subsection{Transcendental-type mappings of slow growth}
\label{sec:DS}
Our final quasiregular map is less well-known. Drasin and Sastry constructed in \cite{DS} a transcendental-type quasiregular map $D:\R^n \to \R^n$ with the property that $D$ behaves like a power mapping on large ring domains. Necessarily the degree of the power mappings has to increase monotonically to obtain a transcendental-type map. We briefly recall some of the elements of their construction that are important for us; see also \cite[Section 3]{NS} for further discussion.

Suppose that $(r_n)_{n\in \N}$ is a sequence of positive real numbers tending to infinity. Let $\nu :\R^+ \to \R^+$ be the piecewise linear map satisfying $\nu(r_n) = n$, for $n \in \N$, and with $\nu'(x)$ defined for $x \in \R^+ \setminus \{r_n : n\in \N\}$. By replacing $(r_n)_{n\in \N}$ with a subsequence if necessary, we may assume that $\nu$ is slowly growing; in particular that it satisfies the conditions of \cite[Theorem 2.1]{DS}. Define
\begin{equation}
\label{eq:Psi} 
\Psi(r) \defeq \exp \left ( \int_1^r \frac{\nu(t) }{t} \: dt \right ).
\end{equation}
Let $E_{\epsilon}$ be as defined in \eqref{Eeq}.

Drasin and Sastry construct a PL version of the sphere $S(1)$, which we denote by $\Delta$. This set has the property that it is the boundary of a convex polyhedron, and that 
\[ \max _{x\in \Delta} |x| = 1.\]
We refer to \cite[p.758-759]{DS} for the precise construction of $\Delta$, which we stress is independent of the sequence $(r_n)_{n \in \N}$ or the choice of $\epsilon$. For $r \geq 0$ we denote by $\Delta(r)$ the topological sphere
\[ \Delta(r) \defeq \{ x\in \R^n : x=ry, y\in \Delta \}.\]

Then $D$ has the property that for each maximal interval $I \subset \R^+ \setminus E_{\epsilon}$, $D$ maps $Q(r)$, for $r\in I$, onto $\Delta(s)$ with $s=\Psi(r)$, recalling \eqref{eq:Psi}. In particular, $D$ behaves like a power mapping on $\{Q(r) : r\in I\}$.

\section{Modifying quasiregular maps}

\subsection{Smoothing out the Drasin-Sastry construction}
Let $(r_n)_{n \in \N}, \Psi, D, \epsilon$, and $E_\epsilon$ be as in Section~\ref{sec:DS}. The Drasin-Sastry map $D$ sends the topological sphere $Q(r)$ onto the topological sphere $\Delta(s)$ for each $r\notin E_{\epsilon}$. To construct specific maximum modulus sets, we need to ``smooth'' this out, so that round spheres are mapped onto round spheres. The first step is straightforward.
\begin{lemma}
\label{lem:3}
There exist bi-Lipschitz maps $\alpha_Q : S(1) \to Q(1)$  and $\alpha_{\Delta}:S(1) \to~\Delta(1)$.
\end{lemma}
\begin{proof}
Both $\alpha_Q$ and $\alpha_{\Delta}$ can be easily constructed via radial maps; we omit the detail.
\end{proof}
We then use Lemma~\ref{lem:3} to construct the following quasiregular map.
\begin{corollary}
\label{cor:1}
There exists a quasiregular map $\widetilde{D} :\R^n \to \R^n$ with the following property. If $r \in \R^+ \setminus E_{\epsilon}$, then $\widetilde{D}$ maps $S(r)$ onto $S(s)$, where $s=\Psi(r)$.
\end{corollary}

\begin{proof}
We radially extend $\alpha_Q$ to a global map by setting 
\[
\alpha_Q(x) \defeq \begin{cases}
|x| \alpha_Q\left( \frac{x}{|x|}\right), &\text{for }x\neq 0, \\
0, &\text{otherwise},
\end{cases}
\] 
and we make the similar extension for $\alpha_{\Delta}$. By an abuse of notation we also call the extended maps $\alpha_Q$ and $\alpha_{\Delta}$. Since each original map is bi-Lipschitz on $S(1)$, and since $Q(1)$ and $\Delta$ are starlike about the origin,  these extended maps are also bi-Lipschitz, and hence quasiconformal.

Then we define a map 
\[
\widetilde{D} \defeq \alpha_{\Delta}^{-1} \circ D \circ \alpha_Q.
\]
By construction, it is evident that $\widetilde{D}$ maps $S(r)$ onto $S(\Psi(r))$ and, as a composition of quasiregular maps, it is quasiregular itself.
\end{proof}

\subsection{A quasiregular map which approximates the identity map}
In this section our goal is to construct a new quasiregular map which, roughly speaking, we will use to decrease the modulus of a quasiregular map in a domain.

For a proper subdomain $U$ of $\R^n$, for $n\geq 2$, we denote by $d_U(x)$ the Euclidean distance from $x$ to $\partial U$. We will use the following elementary lemma.

\begin{lemma}
\label{lem:p}
With $U$ as above, the function $p : U \to \R^+$ given by 
\[ p_U(x) \defeq \frac{d_U(x)}{1+d_U(x)}\] 
is $1$-Lipschitz.
\end{lemma}

\begin{proof}
It is well-known that $d_U$ is a $1$-Lipschitz function. It is easy to check that the function $h(t) = \frac{t}{1+t}$  has derivative at most $1$ for $t\geq 0$. The lemma follows since $p_U = h\circ d_U$.
\end{proof}

A key tool is the following lemma, in which the image of every point in a domain is moved in a certain direction.

\begin{lemma}
\label{lem:1}
Let $n\geq 2$. Then there exists $K>1$, independent of $n$, with the following property. If $U\subsetneq \R^n$ is a domain, then there exists a $K$-quasiconformal surjection $f_U :U \to U$ such that
\[
f_U(x)_n < x_n, \quad\text{for } x \in U,
\]
and $f_U$ extends to the identity on $\partial U$. Moreover, $f_U$ preserves all maximal line segments in $U$ parallel to the $n$-th coordinate axis.
\end{lemma}

\begin{proof}
We define $f_U$ via
\[ 
f_U(x_1,\ldots,x_{n-1} ,  x_n) \defeq \left (x_1,\ldots, x_{n-1}, x_n - \frac{p_U(x)}{2} \right ),
\]
recalling $p_U$ from Lemma \ref{lem:p}.
Clearly $f_U$ satisfies the last two conclusions of the lemma, and so we need to show that $f_U$ is a quasiconformal surjection. 

First we show that $f_U$ is surjective. To see this, if $L \subset U\cap \lambda_{x'}$ is any bounded maximal line segment oriented in the $n$-th coordinate direction, then $f_U$ extends to the identity on the two endpoints of the closure of $L$. Since $f_U(L) \subset L$, and $f_U$ is continuous, the Intermediate Value Theorem implies that $f_U(L) = L$. If $L$ is unbounded, then since $p_U(x) \leq 1$ for all $x\in U$, it follows that
\[ \lim_{x_n \to \pm \infty, x\in L} f_U(x)_n \to \pm \infty,\]
with the appropriate parity. Again we conclude that $f_U(L) = L$.

Next we show that $f_U$ is injective. Suppose, by way of contradiction, that there exist $x, y \in U$ with $x\neq y$ and $f_U(x) = f_U(y)$. Then we have equality in the first $n-1$ coordinates and $x_n - p_U(x)/2 = y_n - p_U(y)/2$. Hence by Lemma \ref{lem:p}
\[
|x-y| = |x_n - y_n| = \frac{|p_U(x) - p_U(y)|}{2} \leq \frac{|x-y|}{2},
\]
which is a contradiction. 

Finally, to see that $f_U$ is quasiconformal, let $x\in U$ and let $r>0$ be small. Denoting by $L_f(x,r)$ the maximum of $|f(x) - f(y)|$ such that $|x-y| =r$ and $\ell_f(x,r)$ the corresponding minimum, one can see by using the fact that $p_U(x)$ is $1$-Lipschitz that
\[
L_f(x,r) \leq \frac{3r}{2},\quad\text{and}\quad \ell_f(x,r) \geq \frac{r}{2}.
\]
Hence $\limsup_{r\to 0} L_f(x,r) / \ell_f(x,r) \leq 3$. Since $x$ here is arbitrary, it follows that $f_U$ has linear distortion bounded everywhere by $3$. Hence, by \cite[Corollary 4]{Gehring}, $f_U$ is $K$-quasiconformal for some $K$ independent of $U$. 
\end{proof}

We apply Lemma~\ref{lem:1} to domains in $\R^n \setminus \{ 0 \}$ in which points are moved radially towards the origin.
First, we need a result on Zorich maps.

\begin{lemma}
\label{lem:zorich}
Let $U\subsetneq \R^n \setminus \{ 0\}$ be a domain. If $\widetilde{U}$ is any connected component of $\mathcal{Z}^{-1}(U)$, then $\mathcal{Z} :\widetilde{U} \to U$ is a surjection.
\end{lemma}

\begin{proof}
By Property~\eqref{branch} of the map $\mathcal{Z}$, we can find $w\in \widetilde{U} \setminus \mathcal{B}_{\mathcal{Z}}$. Set $x = \mathcal{Z}(w)$ and suppose $y\in U$. Since $U$ is a domain, we can find a path $\gamma:[0,1] \to U$ with $\gamma(0) = x$, $\gamma(1) = y$ with the extra condition that $\gamma \cap \mathcal{Z} ( \mathcal{B}_{\mathcal{Z}} )$ is either empty or, possibly, $y$.

It follows that for every $z\in \gamma \setminus \{ y \}$, there exists a neighbourhood $U_z$ and a branch of the inverse $\mathcal{Z}^{-1}$ from $U$ into $\widetilde{U}$. Hence we can uniquely lift $\gamma :[0,1) \to U$ to a path $\widetilde{\gamma}: [0,1) \to \widetilde{U}$. If $Y$ is the accumulation set 
\[ Y {\color{red} \defeq}  \bigcap_{0<t<1} \overline{ \{ \widetilde{\gamma}(s) : s>t \} } ,\]
then, by continuity, any $u\in Y$ must satisfy $\mathcal{Z}(u) = y$. Since quasiregular mappings are open and discrete, it follows that $Y$ must consist of one point, say $u$, in $\widetilde{U}$. Since $\mathcal{Z}(u) = y$, the lemma is proved.
\end{proof}

\begin{corollary}
\label{cor:2}
Let $n\geq 2$. Then there exists $K'>1$ with the following property. If $U \subsetneq \R^n \setminus \{ 0 \}$ is a domain, then there exists a $K'$-quasiconformal map $h_U : U \to U$ such that 
\[
|h_U(x)| < |x|, \quad\text{for } x\in U,
\]
and $h_U$ extends to the identity on $\partial U$.
\end{corollary}

\begin{proof}
Given $U \subsetneq \R^n \setminus \{ 0 \}$, let $\widetilde{U}$ be a (fixed) connected component of $\mathcal{Z}^{-1}(U)$. Suppose there exist $y_1,y_2 \in \widetilde{U}$ and $g\in G$ such that $y_2 = g(y_1)$; here $G$ is the group of isometries from Section~\ref{sec:Z}. Since connected open sets in $\R^n$ are path connected, it follows that $\widetilde{U}$ is completely invariant under the subgroup generated by $g$. Let $G_1$ be the maximal subgroup of $G$ such that $\widetilde{U}$ is completely invariant under $G_1$.
If $y_1,y_2$ are two points of $\widetilde{U}$ in the same orbit of $G_1$, it follows that $d_{\tilde{U}}(y_1) = d_{\tilde{U}}(y_2)$ and hence $p_{\widetilde{U}}(y_1) = p_{\widetilde{U}}(y_2)$.

We define $h_U$ as follows. First, let $f_{\widetilde{U}}$ be the map from Lemma \ref{lem:1}. Now, given $x\in U$, choose $y \in \mathcal{Z}^{-1}(x) \cap \tilde{U}$, which is guaranteed by Lemma \ref{lem:zorich}, and then set 
\[
h_U(x) \defeq \mathcal{Z}(f_{\widetilde{U}}(y)).
\]

To see that this is well-defined, suppose that $y' \in \tilde{U}$ is also an element of $\mathcal{Z}^{-1}(x)$. Then, by the comments in Section~\ref{sec:Z}, $y' = g(y)$ for some $g\in G_1$. Since every element of $G$ preserves $n$-th coordinates, since $\widetilde{U}$ is completely invariant under $G_1$, and since $p_{\tilde{U}}(y) = p_{\tilde{U}}(g(y)) = p_{\tilde{U}}(y')$, it follows that
\[
\mathcal{Z} ( f_{\widetilde{U}}(y') ) = \mathcal{Z} ( f_{\widetilde{U}}(g(y)) ) = \mathcal{Z}(f_{\widetilde{U}}(y)) = h_U(x).
\]
This completes our proof that $h_U$ is well-defined.

Away from 
$\mathcal{Z}(\mathcal{B}_{\mathcal{Z}})$ we can locally define $h_U$ using $\mathcal{Z} \circ f_{\widetilde{U}} \circ \mathcal{Z}^{-1}$. This map is quasiconformal with maximal dilatation bounded above by the product of the square of the maximal dilatation of $\mathcal{Z}$ and the value $K$ from Lemma \ref{lem:1}. In particular, this is independent of $U$.
Since 
$\mathcal{Z}(\mathcal{B}_{\mathcal{Z}})$ consists of lines, it is removable for quasiconformal maps and it follows that $h_U$ is $K'$-quasiconformal with $K'$ independent of $U$. 

Finally, since $f_{\widetilde{U}}(y)_n < y_n$ it follows by Property~\eqref{spheres} of the map $\mathcal{Z}$ that we have $|h_U(x)| < |x|$.
Clearly $h_U$ extends to the identity on $\partial U$.
\end{proof}

\section{Proofs of the main results}

\subsection{Polynomial-type}

\begin{proof}[Proof of Theorem \ref{thm:1}] 
Fix $n\geq 2$, and suppose that $T \subset \R^n$ satisfies the hypotheses of the theorem. First we construct a quasiconformal map $h_1 :\R^n \to \R^n$ as follows. For $x\in T$, we set $h_1(x)=x$. If $U$ is any component of $\R^n \setminus T$, then we set $h_1 \equiv h_U$, where $h_U$ is the map from Corollary \ref{cor:2}. Clearly $|h_1(x)| < |x|$ if and only if $x\notin T$. 

We need to show that $h_1$ is quasiconformal. In fact it is easier to work in a slightly different setting. Set $T' \defeq \mathcal{Z}^{-1}(T)$, and define a map $f$ as the identity on $T'$, and equal to the map $f_U$ from Lemma~\ref{lem:1} on each component $U$ of the complement of $T'$. If $x \notin \mathcal{Z} (\mathcal{B}_{\mathcal{Z}})$, then we can define a branch of $\mathcal{Z}^{-1}$ so that, in a neighbourhood of $x$, we have $h_1 = \mathcal{Z} \circ f \circ \mathcal{Z}^{-1}$. If we can show that $f$ is quasiconformal, then in conjunction with the fact that $\mathcal{Z} ( \mathcal{B}_{\mathcal{Z}})$ has $\sigma$-finite $(n-1)$-dimensional Hausdorff measure and \cite[Theorem 35.1]{Vaisala}, it follows that $h_1$ is quasiconformal.

To prove that $f$ is quasiconformal, first recall that the maximal dilatation of $f_U$ is independent of $U$. Hence it is sufficient to show that the linear distortion of $f$ is uniformly bounded in $T'$. 
%
Fix $x_0 \in T'$, and let $S_r \defeq \{ x : |x-x_0| = r \}$ for $r > 0$. If $x \in S_r \cap T'$, then since $f$ fixes every point of $T'$ we have
\[
|f(x) - f(x_0)| = |x-x_0| = r.
\]
On the other hand, if $x \in S_r \setminus T'$, then it is in some component $U$ of the complement of $T'$. Note that, using the notation of Lemma~\ref{lem:1},
\[
d_U(x) \leq |x-x_0| = r,
\]
and so
\[
p_U(x) \leq r/(1+r).
\]
We can deduce that for all small values of $r$, if $x \in S_r$, then we have both
\[
|f(x) - f(x_0)| \leq |x-x_0| + p_U(x)/2 \leq 3r/2,
\]
and
\[
|f(x)-f(x_0)| \geq |x-x_0| - p_U(x)/2 \geq r/2.
\]

As in the proof of Lemma~\ref{lem:1}, it follows that $\limsup_{r\to 0} L_f(x_0,r) / \ell_f(x_0,r) \leq 3$. Since $x_0$ here is arbitrary, it follows by \cite[Corollary 4]{Gehring} that $f$ is quasiconformal.

Suppose that $d \in \N$ is given. Composing $h_1$ by a power map $P$, as in Section~\ref{sec:power}, yields a quasiregular map $h \defeq P\circ h_1$ of degree $d^{n-1}$, and which satisfies $|h(x) | < |x|^d$ if and only if $x\notin T$. It follows straightforwardly that $\M(f) = T$, as required.
\end{proof}

\subsection{Transcendental-type}

\begin{proof}[Proof of Theorem \ref{thm:2}]
Suppose that $T$ satisfies the hypotheses of the theorem. Suppose that $(r_n)_{n\in \N}$ and $\epsilon \in (0,1)$ are given, and let $E_\epsilon$ and $S_\epsilon$ be the exceptional sets defined in \eqref{Eeq}. Let $D$ be the Drasin-Sastry map from Section~\ref{sec:DS}, which behaves like a power-type map on each component of $\bigcup_{r\notin E_{\epsilon}} Q(r)$. We modify $D$ to $\widetilde{D}$ via Corollary \ref{cor:1}. Then for every $r\notin E_{\epsilon}$, $\widetilde{D}$ maps $S(r)$ onto $S(\Psi(r))$.

Next we construct a quasiconformal map $h_1 :\R^n \to \R^n$ similar to that in the proof of Theorem \ref{thm:1}. In particular, we set $T' = T \setminus S_\epsilon$. We define $h_1$ to be the identity on $T'$, and for any component $U$ of $\R^n \setminus T'$, we set $h_1$ to be $h_U$ via Corollary \ref{cor:2}.

Let $h$ be the transcendental-type quasiregular map $h \defeq \widetilde{D} \circ h_1$. We need to show that $h$ satisfies the conclusions of the theorem. Suppose that $r \in \R$ is such that $r \notin E_\epsilon$, and suppose that $|x| = r$. By construction, if $x \in T$, then $h_1(x) = x$, and so $|h(x)| = \Psi(|x|) = M(r, h)$. On the other hand, if $x \notin T$, then $|h_1(x)| < |x|$, and so $|h(x)| < M(r, h)$. This completes the proof.
\end{proof}

\end{document}